\documentclass[11pt]{amsart}
\usepackage[dvips]{graphicx} 
\usepackage{amssymb,amscd,color,latexsym,epsfig,xypic,hyperref}
\usepackage[mathscr]{euscript}
\usepackage{tikz-cd}
\usepackage{tikz}
\usetikzlibrary{calc,trees,positioning,arrows,chains,shapes.geometric,%
    decorations.pathreplacing,decorations.pathmorphing,shapes,%
    matrix,shapes.symbols}

\tikzset{
>=stealth',
  punktchain/.style={
    rectangle,
    rounded corners,
    draw=black, thick,
    minimum height=3em,
    text centered,
    on chain},
  line/.style={draw, thick, <-},
  element/.style={
    tape,
    top color=white,
    bottom color=blue!50!black!60!,
    minimum width=8em,
    draw=blue!40!black!90, very thick,
    text width=10em,
    minimum height=3.5em,
    text centered,
    on chain},
  every join/.style={->, thick,shorten >=1pt},
  decoration={brace},
  tuborg/.style={decorate},
  tubnode/.style={midway, right=2pt},
}

\DeclareFontFamily{OT1}{rsfs}{}
\DeclareFontShape{OT1}{rsfs}{n}{it}{<-> rsfs10}{}
\DeclareMathAlphabet{\curly}{OT1}{rsfs}{n}{it}

\DeclareFontFamily{U}{mathb}{\hyphenchar\font45}
\DeclareFontShape{U}{mathb}{m}{n}{
      <5> <6> <7> <8> <9> <10> gen * mathb
      <10.95> mathb10 <12> <14.4> <17.28> <20.74> <24.88> mathb12
      }{}
\DeclareSymbolFont{mathb}{U}{mathb}{m}{n}

\makeatletter
\newcommand{\eqnum}{\refstepcounter{equation}\textup{\tagform@{\theequation}}}
\makeatother

\renewcommand\;{\hspace{.6pt}}

\newcommand\C{\mathbb C}
\newcommand\Q{\mathbb Q}

\newcommand\Z{\mathbb Z}

\newcommand\II{\mathbb I}

\newcommand\LL{\mathbb L}

\newcommand\PP{\mathbb P}

\newcommand\cO{\mathcal O}

\newcommand\cE{\mathcal E}

\newcommand\cM{\mathcal M}

\newcommand\cP{\mathcal P}

\newcommand\VW{\mathsf{VW}}
\newcommand\J{\mathsf J}

\renewcommand\({\big(}
\renewcommand\){\big)}

\makeatletter
\newcommand{\so}{\ \ext@arrow 0359\Rightarrowfill@{}{\hspace{3mm}}\ }
\makeatother
\newcommand{\rt}[1]{\xrightarrow{\ #1\ }}
\newcommand\To{\longrightarrow}

\newcommand\ot{\leftarrow}

\newcommand\INTO{\ \ar@{^(->}[r]<-.2ex>}

\newcommand\Mapsto{\ \longmapsto\ }

\renewcommand\_{^{}_}

\newfont{\bigtimesfont}{cmsy10 scaled \magstep5}
\newcommand{\bigtimes}{\mathop{\lower0.9ex\hbox{\bigtimesfont\symbol2}}}
\renewcommand\={\ =\ }

\newcommand\udot{^{\bullet}}

\DeclareMathSymbol{\lefttorightarrow}{3}{mathb}{"FC}
\DeclareMathSymbol{\righttoleftarrow}{3}{mathb}{"FD}

\newcommand\vir{\operatorname{vir}}
\newcommand\red{\operatorname{red}}

\newcommand\coker{\operatorname{coker}}

\newcommand\Hom{\operatorname{Hom}}
\renewcommand\hom{\curly H\!om}

\newcommand\Ext{\operatorname{Ext}}
\newcommand\ext{\curly E\hspace{-.3mm}xt}

\newcommand\Hilb{\operatorname{Hilb}}

\newcommand\beq[1]{\begin{equation}\label{#1}}
\newcommand\eeq{\end{equation}}
\newcommand\beqa{\begin{eqnarray*}}
\newcommand\eeqa{\end{eqnarray*}}

\newcommand\arXiv[1]{\href{http://arxiv.org/abs/#1}{arXiv:#1}}
\newcommand\mathAG[1]{\href{http://arxiv.org/abs/math/#1}{math.AG/#1}}
\newcommand\hepth[1]{\href{http://arxiv.org/abs/hep-th/#1}{hep-th/#1}}

\DeclareRobustCommand{\SkipTocEntry}[3]{}
\makeatletter
\newcommand\@dotsep{4.5}
\def\@tocline#1#2#3#4#5#6#7{\relax
  \ifnum #1>\c@tocdepth 
  \else
    \par \addpenalty\@secpenalty\addvspace{#2}%
    \begingroup \hyphenpenalty\@M
    \@ifempty{#4}{%
      \@tempdima\csname r@tocindent\number#1\endcsname\relax
    }{%
      \@tempdima#4\relax
    }%
    \parindent\z@ \leftskip#3\relax \advance\leftskip\@tempdima\relax
    \rightskip\@pnumwidth plus1em \parfillskip-\@pnumwidth
    #5\leavevmode #6\relax
    \leaders\hbox{$\m@th
      \mkern \@dotsep mu\hbox{.}\mkern \@dotsep mu$}\hfill
    \hbox to\@pnumwidth{\@tocpagenum{#7}}\par
    \nobreak
    \endgroup
  \fi}
\makeatother

\makeatletter \@addtoreset{equation}{section} \makeatother
\renewcommand{\theequation}{\thesection.\arabic{equation}}

\newtheorem{thm}[equation]{Theorem}
\newtheorem{thm*}{Theorem}
\newtheorem{lem}[equation]{Lemma}

\newtheorem{prop}[equation]{Proposition}

\theoremstyle{definition}
\newtheorem{rmk}[equation]{Remark}

\title{Refined sheaf counting on local $K3$ surfaces}
\author{Richard P. Thomas}

\begin{document}
\maketitle
\begin{abstract} We compute all refined sheaf counting invariants\,---Vafa-Witten, reduced DT, stable pairs and Gopakumar-Vafa\,---\,for all classes on local $K3$ surfaces. Along the way we develop rank 0 Vafa-Witten theory on $K3$ surfaces.

An important feature of the calculation is that the ``instanton contri\-bution"---\,of sheaves supported scheme theoretically on $S$\,---\,to any of the invariants depends only on the square of the class, not its divisibility.
\end{abstract}


\addtocontents{toc}{\protect\setcounter{tocdepth}{-1}}
\section*{Introduction}

We study the refined Vafa-Witten invariants $\VW(t)\in\Q[t^{\pm1/2}]$  of \cite{KVW} on a polarised $K3$ surface $(S,\cO_S(1))$. Equivalently\,---\,via the spectral construction relating Vafa-Witten Higgs pairs on $S$ to compactly supported sheaves on the local $K3$ surface $X=S\times\C$ \cite{TT1}\,---\,these are the product of
$$
\frac{t^{1/2}}{(t-1)}\,,\ \text{ where $t$ is the standard weight 1 representation of $\C^*$,}
$$
and the \emph{Nekrasov-Okounkov-twisted\;\footnote{Defined by tensoring the virtual structure sheaf of the moduli space $\cM$ by $K_{\vir}^{1/2}$ \cite{NO}.} $\C^*$-equivariant $K$-theoretic reduced DT invariants} counting \emph{compactly supported sheaves} on $X$.

\begin{thm*} \label{VWt}
\cite[Conjecture 5.43]{KVW} is true: for a Mukai vector $v\in H^*(S,\Z)$,
$$
\VW_v(t)\=\sum_{r|v}\frac{t^{\,v^2/2r-r}\chi\_{-t^r}(\Hilb^{1-v^2/2r^2\!}S)}{[r]_t^2}\,.
$$
\end{thm*}
Here $[r]\_t$ is the quantum integer $(t^{r/2}-t^{-r/2})/(t^{1/2}-t^{-1/2})$ and the Hirzebruch genera $\chi\_{-t}(\Hilb^dS)$ are, by results of G\"ottsche-Soegel \cite{GS}, the coefficients of the unique Jacobi cusp form of index 10 and weight 1,
\beq{jac}
\sum_{d=0}^\infty t^{-d}\chi\_{-t}(\Hilb^dS)\;q^{d-1}\=\frac1q\prod_{d=1}^\infty\frac1{(1-q^d)^{20}(1-tq^d)^2(1-t^{-1}q^d)^2}\,.
\eeq

Theorem \ref{VWt} is standard when the Mukai vector $v$ is primitive, so that any $\C^*$-invariant compactly supported semistable sheaf on $X$ is the push forward of a stable sheaf on $S$, for a generic choice of the polarisation $\cO_S(1)$. Then the sum collapses to $t^{-d}\chi\_{-t}(\Hilb^dS)$ where $v^2=2-2d$. But for general $v$ it seems a minor miracle. It was conjectured in \cite[Conjecture 5.43]{KVW} and even claimed in \cite[Theorem 5.44]{KVW} when $v$ a prime multiple of a primitive class. Martijn Kool pointed out a gap in its proof\;\footnote{There was basically no proof.} which Theorem \ref{VWt} plugs. It also proves \cite[Conjecture 4.6]{GK} and the $K$-theoretic $SU(r)/PSU(r)$ S-duality of \cite[Equation 1]{JK} for prime rank $r$ on the $K3$ surface $S$.

We will see Theorem \ref{VWt} is basically equivalent to (refinements of) a number of other remarkable results about curve and sheaf counting invariants on $K3$ surfaces depending  only on the square of the curve class, \emph{not its divisibility}:\footnote{An example of a more difficult conjecture of this sort is \cite[Conjecture C2]{OP}, proved in some cases in divisibility 2 in \cite{BB} but otherwise still open.}
\begin{itemize}
\item the stable pair invariants of $S$, computed in \cite{KT2},
\item the reduced Gopakumar-Vafa invariants of $X$, computed in \cite{KKV},
\item the DT invariants counting semistable dimension $1$ sheaves on $X$ \nolinebreak(the ``$\chi$-independence conjecture" of Joyce and Toda) computed in \cite{MT}.
\end{itemize}

Specialising Theorem \ref{VWt} at $t^{1/2}=1$ recovers the formulae of \cite{MT, To1} for the numerical invariants of local $K3$ surfaces. Ultimately the proof here is simpler due to progress in wall crossing technology. Because \cite{MT} was underpinned by the Joyce-Song wall crossing formula used in \cite{To1, TodaBPS} the computations had to go via the \emph{compact} Calabi-Yau 3-fold $S\times E$, with $E$ an elliptic curve. Then $E$ had to be degenerated to a rational elliptic curve and normalised to $\PP^1$, to which $\C^*$ localisation could be applied.
Here we use the $K$-theoretic wall crossing formula of Liu and Kuhn-Liu-Thimm \cite{Liu, KLT} which,  for local $K3$ surfaces, has two simplifying advantages:
\begin{itemize}
\item it applies in non-compact and equivariant settings, so we can work directly on $X=S\times\C$, and
\item it uses virtual cycles (in fact virtual structure sheaves) instead of weighted Euler characteristics, so nonvanishing $H^2(\cO_S)$ terms in the obstruction theory for sheaves on $X$ lead to most terms in the wall crossing formula vanishing.
\end{itemize}
Georg Oberdieck has derived related results by similar methods for local Enriques surfaces \cite{Ob3}.

\subsection*{Plan}
Since refined Vafa-Witten theory on $S$ is only defined in positive rank in \cite{KVW} we develop the rank 0 theory in Section \ref{0VW}. (As in \cite{KVW} we use Joyce-Song pairs to handle strictly semistable sheaves.) Later, in Section \ref{inst}, this allows us to use autoequivalences of the derived category of $S$ to turn any sheaf counting problem into one counting \emph{dimension 1 sheaves}.

In Section \ref{VMCF} we explain (and slightly extend) a vanishing theorem and multiple cover formula from \cite{KVW, KKV}. This reduces the computation to one that is scheme theoretically supported on the zero section $S\subset X$. Some wall crossing analysis in Section \ref{inst} changes it from a computation with Joyce-Song pairs of dimension 1 sheaves to one with \emph{stable pairs} \cite{PT1}. 

The upshot is that it is enough to compute all refined reduced stable pair invariants of $S$. These have the advantage (over Joyce-Song pairs) that their moduli space is naturally cut out of a smooth ambient space by a section of a vector bundle \cite{KT1}, making direct calculation with the virtual cycle possible. Indeed all the pieces for this computation were already in the literature, as we review in Section \ref{S}. Results of \cite{KT2} show the computation is blind to the divisibility of the class we are working with. Thus we can work with primitive irreducible classes, quickly reducing the computation to a celebrated calculation of Kawai-Yoshioka \cite{KY}. The final answer gives all refined stable pairs invariants in equations \eqref{MCFp} and \eqref{KY}:
$$
P_{\beta,\chi}(t)\=\sum_{d|(\beta,\chi)}\frac{(-1)^{\chi-\chi/d}}{[d]\_t}P^{\,h_d}_{\frac\chi d}(t^d),
$$
where $2h_d-2=\beta^2/d^2$ and the $P^h_\chi(t)$ are determined by the product formula
$$
\(q^{-1}+[2]\_t+q\)\ \sum_{h=0}^\infty\ \sum_{\chi=1-h}^\infty P^h_\chi(t)\,u^hq^\chi\=\hspace{53mm}\vspace{-2mm}
$$
$$
\scalebox{0.82}{$\displaystyle{\prod_{n\geq1}\frac{1}{\(1+t^{-\frac12}qu^n\)\(1+t^{\frac12}q^{-1}u^n\)\(1+t^{\frac12}qu^n\)\(1+t^{-\frac12}q^{-1}u^n\)(1-u^n)^{18}(1-tu^n)(1-t^{-1}u^n)}\,.}$}\vspace{1mm}
$$

\begin{rmk}
Cohomological DT theory gives yet more refined invariants, very roughly modelled on the Hodge polynomial $\sum h^{i,j}u^iv^j$ of the moduli space\,---\,as studied for $K3$ surfaces in \cite{DHS}, for instance. Specialising to $u=-t,\,v=-1$ gives the $\chi\_{-t}$ refinement of the DT invariant which\,---\,for local Fano surfaces and local K3 surfaces\,---\,is expected to coincide with the Nekrasov-Okounkov refinement considered in this paper.\footnote{This expectation has been proved for local Fano surfaces in \cite[Theorem 5.15]{KVW}. For local $K3$ surfaces it would be a refinement of the equality $\mathsf{vw=VW}$ of \cite[Theorem 7.10]{MT}. (It is not true for general type surfaces \cite[Section 1.11]{TT1} or Enriques surfaces \cite{ObE}.) Note Descombes \cite{De} has proved a general result relating the $\chi\_{-t}$ and Nekrasov-Okounkov refinements of \emph{compact} moduli spaces.} This is different from the $u=v=t^{1/2}$ refinement considered in \cite{KKP}, and indeed the resulting product formulae \eqref{KY} and \cite[Equation 5.1]{KKP}, while formally similar, are certainly different.
\end{rmk}

\noindent\textbf{Acknowledgements.} I am most grateful to Martijn Kool for spotting the gap in \cite{KVW} that lead to this paper, and for very useful correspondence. I thank Henry Liu and Felix Thimm for $K$-theoretic wall crossing discussions, Pierre Descombes, Ben Davison and Georg Oberdieck for refined DT discussions, and Davesh Maulik and Rahul Pandharipande for years of discussions.
The work was supported by a Royal Society research professorship.
\tableofcontents
\vspace{-1cm}

\addtocontents{toc}{\protect\setcounter{tocdepth}{1}}
\section{Vafa-Witten invariants in all ranks}\label{0VW}
The proof of Theorem \ref{VWt} will pass through rank 0 VW invariants on $S$ (corresponding in DT language to sheaves on $X$ whose support has dimension $<2$). We give a brief explanation of how these can be defined (on $K3$ surfaces only), since the definitions in \cite{TT1, TT2, KVW} assumed rank $\ge1$.

\subsection{Stable case}
Fix a Chern character in $H^*_c(X,\Q)$ of compactly supported sheaves $E$ on $X$, or equivalently the Mukai vector $v\in H^*(S,\Z)$ of $p_*\;E$, where $p$ is the projection $X=S\times\C\to S$. For simplicity we begin by defining the invariants when $v$ is chosen so that semistability implies stability. We consider the moduli space $\cM_v$ of compactly supported stable sheaves on $X$ divided by the action of translation,
$$
\cM_v/\C.
$$
Let $\cE$ denote a (possibly twisted) universal sheaf on $\cM_v\times X$ and let $\pi$ be the projection $\cM_v\times X\to\cM_v$. Then the Atiyah class map of \cite{HT}
$$
\(\tau^{[1,2]}R\hom_\pi(\cE,\cE)\)^\vee[-1]\To\LL_{\cM_v}
$$
defines a symmetric perfect obstruction theory on $\cM_v$ with virtual dimension zero and virtual tangent bundle $\tau^{[1,2]}R\hom_\pi(\cE,\cE)[1]$. The additive action of $\C$ defines a nonzero trivial tangent direction in this complex
\beq{1}
\cO_{\cM_v}\otimes t\rt{\partial_t}\ext^1_{\pi}(\cE,\cE)\rt{h^0}\tau^{[1,2]}R\hom_\pi(\cE,\cE)[1].
\eeq
Applying Serre duality and tensoring by $\omega_X^{-1}\cong\cO_X\otimes t$ we also get a trivial cosection of the obstruction sheaf,
\beq{2}
\tau^{[1,2]}R\hom_\pi(\cE,\cE)[1]\To\cO_{\cM_v}.
\eeq
Since the composition of \eqref{1} and \eqref{2} is zero we can remove both terms from the virtual tangent bundle (taking the double cone on them) to produce a perfect obstruction theory on $\cM_v/\C$,
\beq{pot}
\(R\hom_\pi(\cE,\cE)_\perp\)^\vee[-1]\To\LL_{\cM_v/\C}.
\eeq
See \cite[Section 3]{Ob} for details (in the very similar situation where $\C$ is replaced by an elliptic curve). We use the notation $\perp$ because, when the dimension of the sheaves is 2, the above virtual tangent bundle is isomorphic to the one $R\hom_\pi(\cE,\cE)_\perp[1]$ constructed in \cite{TT1} by a different method using the central fibre of a centre of mass map $\cM_v\to\C$ in place of $\cM_v/\C$.

The symmetric perfect obstruction theory \eqref{pot} endows $\cM_v/\C$ with a virtual structure sheaf $\cO_{\cM_v/\C}^{\vir}$ \cite{CFK,FG}. Picking a square root of the virtual cotangent bundle\footnote{While such ``orientation data" has been shown to exist it is simpler to work in $K\(\cM_v,\Z\big[\frac12\big]\)$ where a canonical square root $K_{\vir}^{1/2}$ exists \cite[Lemma 5.1]{OT}.}
$$
K_{\vir}\ :=\ \det\Big(\(R\hom_\pi(\cE,\cE)_\perp\)^\vee[-1]\Big)\=
\det\(R\hom_\pi(\cE,\cE)_\perp\)
$$
and the refined VW invariants are then
\beq{VWdef}
\VW_v(t)\ :=\ \chi\_t\Big(\cO_{\cM_v/\C}^{\vir}\otimes K_{\vir}^{1/2}\Big)
\=\chi\_t\!\left(\frac{\cO_{(\cM_v/\C)^{\C^*}}^{\vir}\otimes K_{\vir}^{1/2}}{\Lambda\udot(N^{\vir})^\vee}\right).
\eeq
Here $\chi\_t$ denotes the $\C^*$-equivariant index\,---\,the character of the virtual representation $H^*\(\cO_{\cM_v/\C}^{\vir}\otimes K_{\vir}^{1/2}\)$ with its natural $\C^*$ action\footnote{The $\C^*$ action on $X$ induces one on $\cM_v$ which commutes with the $\C$ action and so descends to $\cM_v/\C$. The obstruction theory \eqref{pot} is also naturally $\C^*$-equivariant.}\,---\,and the right hand side is produced by the localisation formula. The result is a rational function in $t^{1/2}$ invariant under $t^{1/2}\ot\hspace{-2mm}\to t^{-1/2}$ \cite[Proposition 2.27]{KVW}.

\subsection{Semistable case}\label{ss}
In general the possible presence of semistable sheaves makes $\cM_v$ a stack, so we rigidify as in \cite{KVW} by using Joyce-Song pairs $(E,s)$. Here
\begin{itemize}
\item $E$ is a compactly supported sheaf on $X$,
\item $s\in H^0(E(n))$
\end{itemize}
for some fixed $n\gg0$, satisfying the stability conditions
\begin{itemize}
\item $E$ is semistable,
\item $s$ factors through no semi-destabilising proper subsheaf $F\subset E$.
\end{itemize}
For fixed Mukai vector $v\in H^*(S,\Z)$ of the sheaves $p_*\;E$, there is a fine moduli space $\cP_{v,n}$ of these Joyce-Song pairs. It is only quasi-projective but its $\C^*$ fixed locus is projective. It also carries a commuting $\C$ action by translation.

A Joyce-Song pair determines a 2-term complex $I\udot:=\{\cO(-n)\to E\}$ whose trace-free self-Exts govern the deformation-obstruction theory of $\cP_{v,n}$. Over $\pi\colon X\times\cP_{v,n}\to\cP_{v,n}$ there is a universal sheaf $\cE$ and universal pair $\II\udot=\{\cO_{X\times\cP_{v,n}}\to\cE\}$ whose Atiyah class defines a perfect obstruction theory
$$
R\hom_\pi(\II\udot,\II\udot)_0^\vee[-1]\To\LL_{\cP_{v,n}}.
$$
It is symmetric of virtual dimension zero. The additive action of $\C$ defines a nonzero trivial tangent direction
$$
\cO_{\cP_{v,n}}\otimes t\rt{\partial_t}\ext^1_{\pi}(\II\udot,\II\udot)\_0\rt{h^0}R\hom_\pi(\II\udot,\II\udot)\_0[1].
$$
Applying Serre duality and tensoring by $\omega_X^{-1}\cong\cO_X\otimes t$ we also get a trivial cosection of the obstruction sheaf,
\beq{vec}
R\hom_\pi(\II\udot,\II\udot)\_0[1]\To\cO_{\cP_{v,n}}.
\eeq
Removing both by double cone as before gives a perfect obstruction theory on $\cP_{v,n}/\C$,
$$
\(R\hom_\pi(\II\udot,\II\udot)\_\perp\)^\vee[-1]\To\LL_{\cP_{v,n}/\C}.
$$
When rank$\;(p_*\;E)\ge1$ this agrees with the perfect obstruction theory of \cite{TT2} constructed using the central fibre of a centre-of-mass map $\cP_{v,n}\to\C$.

Just as above the perfect obstruction theory endows $\cP_{v,n}/\C$ with a virtual structure sheaf $\cO_{\cP_{v,n}/\C}^{\vir}$, while orientation data gives a square root of the virtual cotangent bundle
$$
K_{\vir}\ :=\ \det\(R\hom_\pi(\II\udot,\II\udot)_\perp\).
$$
Thus we get refined pairs invariants 
$$
P_v(n,t)\ :=\ \chi\_t\Big(\cO_{\cP_{v,n}/\C}^{\vir}\otimes K_{\vir}^{1/2}\Big)\=\chi\_t\left(\frac{\cO_{(\cP_{v,n}/\C)^{\C^*}}^{\vir}\otimes K_{\vir}^{1/2}}{\Lambda\udot(N^{\vir})^\vee}\right).
$$
It was conjectured in \cite{KVW} and proved in \cite{Liu} that
\beq{indep}
P_v(n,t)\,/\,[\chi(v(n))]\_t \ \text{ is independent of }n
\eeq
so that we may define the refined Vafa-Witten invariants by\footnote{The (much more complicated) general formula simplifies here because $H^2(\cO_S)\ne0$: see Section \ref{wc} for an explanation.}
\beq{conj}
P_v(n,t)\=(-1)^{\chi(v(n))-1}\big[\chi(v(n))\big]_t\VW_v(t).
\eeq
When semistability and stability coincide this recovers the previous definition \eqref{VWdef} by \cite[Proposition 5.5]{KVW}.
 
\section{Vanishing and multiple cover formula}\label{VMCF}
By \cite[Corollary 1]{KKV} and \cite[Theorem 5.34]{KVW} the only $\C^*$-fixed pairs which contribute to the invariants (\ref{VWdef},\ \ref{conj}) are ``\emph{uniformly thickened}" in the $\C$ direction: they are determined by a pair $(E,s)$ on $S$ and a thickness $d$ defining the pair
\beq{unif}
\(p^*E\otimes\cO_{dS},\,(p^*s)|_{dS}\)
\eeq
supported on $dS\subset X$\,---\,the scheme-theoretic thickening of the zero section $S\subset X$ defined by the $d$th power of its ideal.

\begin{rmk}
For all other pairs there is a nowhere vanishing vector field on $\cP_{v,n}/\C$ which, by Serre duality and the symmetry of the obstructon theory, defines a nowhere zero cosection of the obstruction sheaf, forcing the invariants to vanish. The vector field arises from a flow which translates the ``$d$th layer" of the underlying sheaf along the $\C$ direction while fixing $p^*E\otimes\cO_{(d-1)S}$. The induced dual cosection is proportional to \eqref{vec} (and so zero on the double cone) precisely on the uniformly thickened pairs. See \cite[Section 5]{KKV} or \cite[Section 5.2]{KVW}.
\end{rmk}

For each fixed $d$ the construction \eqref{unif} makes the moduli space of $d$-times uniformly thickened pairs isomorphic to the moduli space of Joyce-Song pairs supported scheme-theoretically on $S$ with Mukai vector $v/d$,
\beq{uniform}
\left(\cP_{(v/d)^d,\,n}\,\big/\C\right)^{\!\C^*}\ \cong\ \(\cP_{(v/d)^1,\;n}/\C\)^{\C^*}.
\eeq
Here we use the obvious partition notation $(v/d)^d=(v/d,v/d,\dots,v/d)$ of \cite[Section 5.4]{KVW} to specify how the charge $v$ is distributed over the layers of the thickening $dS$. The $\C^*$-fixed perfect obstruction theory on the left hand side of \eqref{uniform} is \emph{not the same} as the one on the right hand side, but is determined by it by a simple formula \cite[Proposition 4]{KKV}, \cite[Section 5.4]{KVW}. This quickly gives
\beq{MCFP}
(-1)^{\chi(v(n))-1}P_{(v/d)^d}(n,t)\=\frac{(-1)^{\chi\(\frac1dv(n)\)-1}}{[d]\_t}P_{(v/d)^1}\(n,t^d\)
\eeq
by \cite[Equation 5.37]{KVW}. As in \cite[Theorem 5.38]{KVW} this $1/[d]\_t$ \emph{multiple cover formula for pairs} \eqref{MCFP} induces the following $1/[d]^2_t$ \emph{multiple cover formula for sheaves} \eqref{MCF}.

\begin{thm}\label{numer}
The summands of the decomposition by thickening type
\beq{PP}
P_v(n,t)\=\sum_{d|v}P_{(v/d)^d}(n,t)
\eeq
each individually satisfy \eqref{indep}. Thus $\VW_v(t)\=\sum_{d|v}\VW_{(v/d)^d}(t)$ similarly decomposes by thickening type with the summands defined by
\beq{VWv}
P_{(v/d)^d}(n,t)\=(-1)^{\chi(v(n))-1}\big[\chi(v(n))\big]_t\VW_{(v/d)^d}(t)
\eeq
and satisfying the multiple cover formula
\beq{MCF}
\VW_{(v/d)^d}(t)\=\frac1{[d]_t^2}\VW_{(v/d)^1}(t^d).
\eeq
\end{thm}

\begin{proof}
We prove the result for multiples $rv_0$ of a primitive class $v_0$ by induction on $r$. For $r=1$ the sum \eqref{PP} collapses to the $d=1$ term for which Liu's proof \cite{Liu} of \eqref{indep} gives the result. So to prove the result for $v=rv_0$ we may assume inductively that it has been proved for smaller multiples.

We adopt the notation that $\pm$ denotes the sign $(-1)^{\chi(w(n))-1}$ when it precedes a pairs invariant with total Mukai vector $w$. Using \eqref{MCFP} we rewrite \eqref{PP} as 
\begin{eqnarray} \nonumber
\pm P_v(n,t) &=& \pm P_{(v)^1}(n,t)+\sum_{d|v,\,d\ge2}\frac{\pm P_{(v/d)^1}\(n,t^d\)}{[d]\_t} \\ \nonumber
 &=& \pm P_{(v)^1}(n,t)+\sum_{d|v,\,d\ge2}\frac{\big[\chi\(\frac1dv(n)\)\big]\_{t^d}\,\VW_{(v/d)^1}(t^d)}{[d]\_t} \\
 &=& \pm P_{(v)^1}(n,t)+\sum_{d|v,\,d\ge2}[\chi(v(n))]\_t\,\frac{\VW_{(v/d)^1}(t^d)}{[d]_t^2}\,, \label{PPP}
\end{eqnarray}
where in the second line we have used the induction assumption on $P_{(v/d)}(n,t)$ and in the third line the identity $[\chi]_{t^d}=[d\chi]\_t/[d]\_t$. Since, by \cite{Liu}, $P_v(n,t)$ satisfies \eqref{indep} we see that every summand in \eqref{PPP} has the same property. Thus every summand in \eqref{PP} does too, and the induction continues. Finally, substituting \eqref{VWv} into \eqref{PPP} gives
$$
\VW_v(t)\=\sum_{d|v}\VW_{(v/d)^d}(t),
$$
where
\begin{equation*}
\VW_{(v/d)^d}(t)\=\frac1{[d]_t^2}\VW_{(v/d)^1}(t^d).\qedhere
\end{equation*}
\end{proof}

\section{The instanton contribution}\label{inst}
By \eqref{MCF} the Vafa-Witten invariants $\VW_v(t)$ are determined by the invariants $\VW_{(v)^1}(t)$ which count semistable sheaves supported \emph{scheme theoretically on $S$} (the so-called ``instanton contribution") via the formula \eqref{VWv} with $d=1$. So to prove Theorem \ref{VWt} what remains to show is
\beq{left}
\VW_{(v)^1}(t)\=t^{\;v^2/2-1}\,\chi\_{-t}\;(\Hilb^{1-v^2/2}S).
\eeq
This is standard if $v$ is primitive (so that semistability and stability coincide for a generic polarisation and the moduli space of stable sheaves supported scheme theoretically on $S$ is deformation equivalent to a Hilbert scheme of points on $S$) but it is remarkable otherwise. It is equivalent, notice, to the invariants depending only on the square of $v$, not its divisibility.

\subsection{Wall crossing}\label{wc}
Before we use wall crossing formulae we make some vague remarks about their general shape, beginning with Joyce-Song's wall crossing formula \cite{JS} for numerical generalised DT invariants on projective Calabi-Yau 3-folds. For simplicity we ignore the Behrend function and work with the ``naive" DT invariants $(-1)^{\dim M}e(M)$ of moduli spaces $M$ of stable objects (generalised to the semistable case by taking the signed Euler characteristic of Joyce's Hall algebra logarithm). The resulting formulae also hold without making these simplifying assumptions.

It is also important to us that the theory can be applied to moduli spaces of (Joyce-Song or stable) pairs, where the DT invariants should be thought of as counting stable complexes $I\udot\in D(\mathrm{coh}\,X)$.

Denote by $\J^{\pm}(v)$ the generalised DT invariants counting semistable objects in class $v=v_1+v_2$ just above and below a wall in a space of stability conditions over which the slopes or phases of $v_1,\,v_2$ cross.

Below the wall we have stable objects $E$ sitting in nontrivial extensions \beq{12}
0\To E_2\To E\To E_1\To0,
\eeq
with $E_i$ stable of class $v_i$. Above the wall these are unstable, but nontrivial extensions in the opposite direction 
\beq{21}
0\To E_1\To E'\To E_2\To0
\eeq
become stable. While we cannot control the Euler characteristics $ext^1(E_i,E_j)$ of the spaces $\PP(\Ext^1(E_i,E_j))$ of extensions (\ref{12},\,\ref{21}), we do know their \emph{difference}. Using stability ($\Hom(E_1,E_2)=0=\Hom(E_2,E_1)$), Serre duality ($\Ext^i(E_1,E_2)=\Ext^{3-i}(E_2,E_1)^*$) and Riemann-Roch,
$$
ext^1(E_2,E_1)-ext^1(E_1,E_2)\=\chi(v_1,v_2),
$$
where $\chi(\ ,\ )$ is the Mukai pairing. Integrating this difference (in the sense of taking weighted Euler characteristic, and further inserting the $(-1)^{\dim}$ sign) over the space of $E_1$s and $E_2\;$s then gives the leading term in the wall crossing formula
\beq{eWCF}
\J^+(v)\=\J^-(v)+\sum_{v_1+v_2=v}(-1)^{\chi(v_1,v_2)-1}\chi(v_1,v_2)\,\J^-(v_1)\,\J^-(v_2)+\dots
\eeq
The higher order terms $\dots$ are a finite sum of terms
$$
C(v_1,\dots,v_k)\,\J^-(v_1)\cdots\J^-(v_k),\ \text{ where }\ v_i+\dots +v_k\,=\,v,
$$
and correspond to semi-destabilising filtrations with $k$ graded pieces in classes $v_i$. So the $k=2$ case is what we described above, in the special case that the $E_1,\,E_2$ are stable on the wall (if they are only semistable their destabilising filtrations will also give rise to longer filtrations and so higher order terms).\medskip

There are now variants \cite{J,Liu,KLT} of the wall crossing formula \eqref{eWCF} using virtual cycles\,---\,or virtual structure sheaves\,---\,in place of (Behrend-weighted) Euler characteristics. (In fact the paper \cite{J} does not work in the Calabi-Yau threefold setting, but  can be made to work there with heavy modification \cite{KLT}.) For DT invariants on a compact Calabi-Yau 3-fold they give the same result, but for \emph{reduced} DT theory on a compact Calabi-Yau 3-fold $X$ with $H^2(\cO_X)\ne0$ they simplify enormously.

The point is that the deformation-obstruction theory of the associated graded of a length $k$ destabilising filtration ($E_1\oplus E_2$ in our $k=2$ example above) has $k$ copies of $H^2(\cO_X)\ne0$ in its obstruction space. Only 1 of these is removed by taking the reduced obstruction theory, so for $k>1$ we end up with a nonzero trivial piece in the obstruction theory, which makes the contribution of the associated graded to the virtual cycle vanish. Ultimately this causes the $k>1$ terms in \eqref{eWCF} vanish, leaving the following trivial wall crossing formula for \emph{reduced} DT invariants
\beq{=WCF}
\J^+_{\red}(v)\=\J^-_{\red}(v).
\eeq
Our situation is a little different because the cosection \eqref{vec} \emph{vanishes} in one (and only one\footnote{In fact it vanishes precisely for (complexes of) sheaves pulled back from $S$ to $X=S\times\C$, but these are not compactly supported. The only one which features in our analysis is the line bundle $\cO_X(-n)$ used to make Joyce-Song pairs (or $\cO_X$ for stable pairs).}) situation: when the complex $I\udot$ is just the line bundle\footnote{This is for Joyce-Song pairs. For stable pairs we take $n=0$.} $\cO_X(-n)$ (i.e. when $I\udot$ is the complex associated to the zero pair $(E,s)=(0,0)$). So there is no reduced obstruction theory for this line bundle. 

Now Vafa-Witten theory uses a reduced obstruction theory for compactly supported sheaves (with the $H^2(\cO_S)=\C$ piece \eqref{2} removed from the $R\Hom(E,E)$ obstruction theory) and for non-trivial Joyce-Song and stable pairs $I\udot$ (reduced by removing the $H^2(\cO_S)=\C$ piece \eqref{vec} from the $R\Hom(I\udot,I\udot)\_0$ obstruction theory). Thus we get the trivial wall crossing formula \eqref{=WCF} \emph{except in one situation}: when $k=2$ and the line bundle $\cO_X(-n)$ is one of the two factors in the associated graded\footnote{In fact $E_2$ will be the shift $F[-1]$ of a compactly supported sheaf because the destabilising filtration will come from the exact triangle $F[-1]\to I\udot\to\cO_X(-n)$ given by rotating the natural triangle $I\udot\to\cO_X(-n)\to F$.}
\beq{gr}
\cO_X(-n)\oplus E_2.
\eeq
The reduced (or Vafa-Witten) obstruction theory of the sum \eqref{gr} is made up of the \emph{unreduced} trace-free obstruction theory
\beq{zro}
R\Hom\(\cO_X(-n),\cO_X(-n)\)_0\=0
\eeq
of the first summand, the \emph{reduced} obstruction theory of the second, and the cross terms $R\Hom(\cO_X(-n),F)=H^*(F(n))$ and $R\Hom(F,\cO_X(-n))=H^*(F(n))^\vee[-3]$.
By \eqref{zro} the first summand is rigid and contributes $+1$ to $\J^-(v_1)$ in the degree 2 term of \eqref{eWCF}, which becomes\footnote{Doing wall crossing in more general classes, some of which admit reduced obstruction theories and some of which have no $H^2(\cO_X)$ terms in their obstruction theory, we should expect a wall crossing formula of the shape 
$$\J^+_{\red}(v)\ =\ \J^-_{\red}(v)+\sum_{v_1+\dots+v_n=v}C(v_1,\dots,v_n)\,\J^-_{\red}(v_1)\;\J^-(v_2)\dots\J^-(v_n)$$ with universal constants $C(v_1,\dots,v_n)$ and \emph{precisely one reduced invariant} in each product.}
$$
\J^+_{\red}(v)\=\J^-_{\red}(v)+(-1)^{\chi(v_2(n))-1}\;\chi\(\cO_X(-n),v_2\)\,\J^-_{\red}(v_2)
$$
when $v=[\cO_X(-n)]+v_2$ and the slopes of $\cO_X(-n)$ and $v_2$ cross over the wall. Working $\C^*$-equivariantly with $K$-theoretic invariants refines this to 
\beq{KWCF}
\J^+_{\red}(v,t)\=\J^-_{\red}(v,t)+(-1)^{\chi(v_2(n))-1}\big[\chi(v_2(n))\big]_t\,\J_{\red}^-(v_2,t),
\eeq
replacing the Euler characteristic $\chi(v_2(n))$ of $\PP\(H^0(F_2(n))\)$ with its $\chi\_{-t}$ genus $[\chi(v_2(n))]\_t\;$.
While most of this discussion has been vague and purely motivational, versions of \eqref{KWCF} are precisely what have been proved in \cite{Liu, KLT}. In particular the reader should recognise \eqref{conj} as an instance of \eqref{KWCF}.

\subsection{Deformations, autoequivalences and wall crossing}
We now follow \cite[Section 4.10]{To1} and \cite[proof of Equation (121)]{To1}, applying deformations of $S$ and autoequivalences of $D(\operatorname{coh}S)$ to turn $v$ into the Mukai vector $(0,\beta,\chi)$ of a \emph{1-dimensional sheaf} with the same square $v^2=\beta^2$ and divisibility.

A complication is that an autoequivalence $\Phi$ will not in general preserve the line bundle $\cO_X(-n)$, so will not preserve the Joyce-Song pairs functor
$$
E\Mapsto F_n(E)\ :=\ \Hom(\cO_X(-n),E).
$$
Instead we have the functor
$$
E\Mapsto G_n(E)\ :=\ \Hom\(\Phi(\cO_X(-n)),E\)
$$
such that $G_n(\Phi(E))=F_n(E)$. While Liu's wall crossing formalism is phrased in terms of the $F_n$, it could equally use $G_n$ (as indeed Joyce allows in his wall crossing formula for cohomological virtual cycles \cite[Section 5.1]{J}) without changing the Vafa-Witten invariants or any of the results; cf. \cite[Remark 5.10(d)]{J}.

Such operations also change the stability condition for sheaves, but we can then wall cross back\footnote{After possibly applying derived duals to get back into the same component of the space of stability conditions, as in the proof of Corollary 4.30 in \cite{To1}.} to Gieseker stability using the \cite{KLT} wall crossing formula. Since $H^2(\cO_S)\ne0$ this is very simple, showing as in \eqref{=WCF} that (equivariant $K$-theoretic) invariants are the same on both sides of the wall.

The upshot is that $\VW_v(t)=\VW_{\Phi_*(v)}$ and to prove \eqref{left} we may, and do, assume from now on that $v=(0,\beta,\chi)$. And since \eqref{left} is true for primitive classes, we only need prove that the $\VW_{(0,\beta,\chi)^1}(t)$ are independent of the divisibility of the Mukai vector $(0,\beta,\chi)$. In fact a weaker statement will be enough for us.

\begin{prop} \label{prop}
To prove Theorem \ref{VWt} it is enough to show $\VW_{(0,\beta,\chi)^1}(t)$ depends only on $\beta^2$, not on $\chi$ nor the divisibility of $\beta.\hfill\square$
\end{prop}

To prove the $\VW_{(0,\beta,\chi)^1}(t)$ do indeed depend only on $\beta^2$ we will calculate them. To do so we relate these sheaf counts to counts of different types of pairs made from them\,---\,\emph{stable pairs} \cite{PT1} rather than Joyce-Song pairs. That is we consider pairs $(F,s)$ on $X$ with $s\in H^0(F)$ satisfying the stability condition
\begin{itemize}
\item $F$ is a \emph{pure} $1$-dimensional sheaf, and 
\item $\dim(\coker s)=0$.
\end{itemize}
There is a fine quasi-projective moduli space $P_\chi(X,\beta)$ of stable pairs with $c_1(p_*\;F)=\beta$ and $\chi(F)=\chi$. It has a projective $\C^*$-fixed locus. Associated to such a stable pair is the 2-term complex $I\udot=\{\cO_X\rt s F\}$. Then the perfect obstruction theory and reduced $K$-theoretic stable pair invariants $P_{\beta,\chi}(t)$ of $X$ are defined in the same way as the Joyce-Song pair invariants in Section \ref{ss}.

We use Toda's wall crossing analysis \cite[Section 5.7]{TodaBPS} in moving from a weak stability condition in which the complexes $I\udot\in P_\chi(X,\beta)$) are stable to one in which their derived duals $(I\udot)^\vee\in P_{-\chi}(X,\beta)$ are stable \cite[Proposition 5.4]{TodaBPS}. This involves crossing the wall $W_\chi$ over which the tautological exact triangle
$$
F[-1]\To I\udot\To\cO_X\rt sF
$$
destabilises $I\udot\in P_\chi(X,\beta)$, then crossing $W_{-\chi}$ where the derived dual
$$
\cO_X\To(I\udot)^\vee\To\ext^2(F,\cO_X)[-1]
$$
in $P_{-\chi}(X,\beta)$ becomes stable. The associated gradeds for these filtrations are of the form \eqref{gr}. The walls in between have associated gradeds which are not of this form so all their pieces have nonzero cosection \eqref{2} or \eqref{vec}. Thus by (\ref{=WCF}, \ref{KWCF}) the upshot is that the difference between the invariants $P_{\beta,\chi}(t)$ counting stable objects above $W_\chi$ and those $P_{\beta,-\chi}(t)$ counting stable objects below $W_{-\chi}$ is
\beq{Kth}
P_{\beta,\chi}(t)-P_{\beta,-\chi}(t)\=(-1)^{\chi-1}[\chi]\_t\;\VW_{(0,\beta,\chi)}(t).
\eeq
Because $H^2(\cO_S)\ne0$ this is much simpler than Toda's wall crossing formula \cite[Theorem 5.7]{TodaBPS} deduced from Joyce-Song.

Analogously to \eqref{MCFP} we also have a multiple cover formula for stable pairs. In fact the first step\,---\,the equality \eqref{uniform} of moduli spaces of Joyce-Song pairs of type $(v)^1$ and of uniform type $(v/d)^d$ was actually proved first in the stable pair setting in \cite[Lemma 1]{KKV}, as was the comparison of the two resulting perfect obstruction theories \cite[Proposition 4]{KKV}. Using these the comparison of the refined invariants of \cite[Section 5.4]{KVW} proceeds in exactly the same way to give the following analogue of \eqref{MCFP},
\beq{MCFp}
(-1)^{\chi-1}P_{(\beta/d,\,\chi/d)^d}(t)\=\frac{(-1)^{\chi/d-1}}{[d]\_t}P_{(\beta/d,\,\chi/d)^1}\(t^d\),
\eeq
using the same partition notation as before.
So we can rerun the proof of Theorem \ref{numer} for $P_{\beta,\chi}(t)-P_{\beta,-\chi}(t)$ and (\ref{Kth},\ \ref{MCFp}) in place of $P_v(n,t)$ and (\ref{conj},\ \ref{MCFP}) to show the wall crossing formula \eqref{Kth} also holds for the $(\ )^1$-contributions to $P_{\beta,\chi}(t)$ supported scheme theoretically on $S$.

\begin{prop}
\beq{vw1}
P_{(\chi,\beta)^1}(t)-P_{(-\chi,\beta)^1}(t)\=(-1)^{\chi-1}[\chi]\_t\;\VW_{(0,\beta,\chi)^1}(t).
\eeq
\end{prop}

\begin{proof}
Expanding both sides of \eqref{Kth} as a sum over uniform partitions and substituting in the stable pairs multiple cover formulae \eqref{MCFp} on the left and the sheaf multiple cover formula \eqref{MCF} on the right gives
$$
\scalebox{.92}{$\displaystyle{\sum_{d|(\beta,\chi)}\!\!\frac{P_{(\beta/d,\,\chi/d)^1}\(t^d\)-P_{(\beta/d,-\chi/d)^1}\(t^d\)}{[d]\_t} \\
\=(-1)^{\frac\chi d-1}\!\!\sum_{d|(\beta,\chi)}\!\left[\frac\chi d\right]_{t^d}\!\frac{\VW_{(0,\beta/d,\,\chi/d)^1}\(t^d\)}{[d]\_t}}$}
$$
by the identity $[\chi/d]_{t^d}=[\chi]\_t/[d]\_t$. Then \eqref{vw1} follows from equating the individual terms in the sum by a descending induction on $d$: the base case of primitive $(\beta/d,\,\chi/d)$ follows already from \eqref{Kth}.
\end{proof}

\section{Computation on $S$}\label{S}
So, combining \eqref{MCF} and \eqref{vw1}, everything is now determined by the invariants $P_{(\chi,\beta)^1}(t)$ counting stable pairs \emph{supported on $S$}. By \cite[Proposition 2.31]{KVW} these are
\beq{hirz}
P_{(\chi,\beta)^1}(t)\=(-1)^{\chi-1}\;t^{\;-(\beta^2+\chi+1)/2}\;\chi^{\vir}_{-t}\(P_\chi(S,\beta)\)
\eeq
and can be computed using the results of \cite{KT1,KT2,KY}. Fix the notation
\begin{itemize}
\item $|\beta|:=\PP(H^0(L))$ is the complete linear system of the unique line bundle $L$ with $c_1(L)=\beta$,
\item $s\in H^0\(S\times|\beta|,\,L\boxtimes\cO(1)\)$ cuts out the universal curve,
\item $Z\subset S\times\Hilb^n(S)$ is the universal subscheme, and
\item $\xymatrix@=10pt{& S\times\Hilb^n(S)\times|\beta| \ar[dl]^(.45){\rho_n}\ar[dr]+<-15pt,8pt>_-\rho \\
\Hilb^n(S)\times|\beta|\hspace{-2mm} && \hspace{-3mm}S\times|\beta| & \hspace{-2mm}\text{are the projections.}}$ 
\end{itemize}
By \cite[Proposition B.8]{PT3} $P_\chi(S,\beta)$ is the relative Hilbert scheme of $\chi+\beta^2/2$ points on the universal curve over $|\beta|$. In \cite{KT1} it is shown that\footnote{Since $H^1(\cO_S)=0$ the Hilbert scheme of divisors of a fixed topological type is just a linear system, making the analysis in \cite{KT1} much simpler.}
$$
P_\chi(S,\beta)\ \subset\ \Hilb^n(S)\times|\beta|, \quad n\,=\,\chi+\beta^2/2,
$$
is cut out by the section $s^{[n]}:=\rho_{n*}\(\rho^*s|_{Z\times|\beta|}\)$ of the vector bundle
$$
E\ :=\ \(L\boxtimes\cO(1)\)^{[n]}\ :=\ \rho_{n*}\!\left(\rho^*(L\boxtimes\cO(1))\big|_{Z\times|\beta|}\right)
$$
and that this description induces the correct perfect obstruction theory on $P_\chi(S,\beta)$. Thus the pushforward of $[P_\chi(S,\beta)]^{\vir}$ to $\Hilb^nS\times|\beta|$ is the Euler class of $E$, so we can compute integrals against cohomology classes $\sigma$ restricted from $\Hilb^n(S)\times|\beta|$ by
\beq{eE}
\int_{[P_\chi(S,\beta)]^{\vir}}\sigma\=\int_{\Hilb^n(S)\times|\beta|}\sigma\cup e(E).
\eeq
In our case \eqref{hirz} we let $\sigma$ be the Hirzebruch genus of the $K$-theory class
$$
T_{\Hilb^n(S)\times|\beta|}-E,
$$
whose restriction to $P_\chi(S,\beta)$ is $T^{\vir}_{P_\chi(S,\beta)}$.
The right hand side of \eqref{eE} is computed in \cite{KT2} by 
pushing down to an integral of tautological classes on $\Hilb^n(S)$ which are then converted to integrals over $S^n$ by the method of \cite{EGL}. Expressing the result in terms of integrals over $S$ writes the invariants \eqref{hirz} as formulae of $\chi$ and $\beta^2$ but \emph{not on the divisibility of }$\beta$, as in \cite[Theorem 1.1]{KT2}.

Thus, to compute \eqref{hirz}, we may assume $\beta$ is primitive. Then $P_\chi(S,\beta)$ is smooth \cite[Proposition C.2]{PT3} and its $\chi\_{-t}$ genus was computed by Kawai-Yoshioka in \cite{KY}. Set $\beta^2=2h-2$ and let
\beq{tnorm}
P^h_\chi(t)\ :=\ (-1)^{\chi-1\;}t^{\;(1-2h-\chi)/2}\chi\_{-t}(P_\chi(S,\beta))\=P_{(\chi,\beta)^1}(t)
\eeq
be the invariants \eqref{hirz}. Then substituting  $\tilde t=1,\,q=u$ and $y=-qt^{-1/2}$ into \cite[Equation 5.159]{KY} gives
\beq{KY}
\(q^{-1}+[2]\_t+q\)\ \sum_{h=0}^\infty\ \sum_{\chi=1-h}^\infty P^h_\chi(t)\,u^hq^\chi\=\hspace{4cm}\vspace{-1mm}
\eeq
$$
\scalebox{0.82}{$\displaystyle{\prod_{n\geq1}\frac{1}{\(1+t^{-\frac12}qu^n\)\(1+t^{\frac12}q^{-1}u^n\)\(1+t^{\frac12}qu^n\)\(1+t^{-\frac12}q^{-1}u^n\)(1-u^n)^{18}(1-tu^n)(1-t^{-1}u^n)}\,.}$}\vspace{1mm}
$$

\begin{lem} The infinite product \eqref{KY} can be written as a sum
\beq{bps}
\sum_{h=0}^\infty\ \sum_{g=0}^h\,n^h_g(t)\Big(q^{-1}+[2]\_t+q\Big)^gu^h,
\eeq
where $n^h_g(t)\in\Q\big[t^{\pm\frac12}\big]$ is a palindromic\;\footnote{\;I.e. invariant under $\tau\ot\hspace{-2mm}\to\tau^{-1}$. The $n_g^h(t)$ are the Nekrasov-Okounkov refinement of the reduced Gopakumar-Vafa invariants of $X$.} Laurent polynomial in $\tau:=t^{1/2}$.
\end{lem}

\begin{proof}
Expanding the product in powers of $u$ we see the coefficient of $u^h$ is a palindromic Laurent polynomial in $q$ of degree $\le h$ (with coefficients which are palindromic Laurent polynomials in $\tau$ of degree $\le h$). The result then follows from the observation that the $\(q^{-1}+[2]\_t+q\)^g$ form a $\C[\tau^{\pm1}]$-basis for the space of these polynomials.

More explicitly, set $n^h_h(t)$ to be the coefficient of $q^h$; this is a palindromic Laurent polynomial in $\tau$. Now subtract
$$
n^h_h(t)\(q^{-1}+[2]\_t+q\)^h
$$
from the product to get a new palindromic Laurent polynomial in $q$ of degree $\le h-1$. Next we let $n^h_{h-1}(t)$ be its $q^{h-1}$ coefficient and subtract $n^h_{h-1}(t)\(q+[2]\_t+q^{-1}\)^{h-1}$. Repeating until we reach $n^h_0(t)$ gives the formula claimed.
\end{proof}

Thus we find
$$
\sum_{\chi=1-h}^\infty P^h_\chi(t)\,q^\chi\=
\sum_{g=0}^h\,n^h_g(t)\(q^{-1}+[2]\_t+q\)^{g-1}.
$$
The bracketed terms on the right are palindromic Laurent polynomials in $q$ for $g\ge1$. Therefore replacing $P^h_\chi(t)$ by $P^h_\chi(t)-P^h_{-\chi}(t)$ on the left hand side (thus replacing the coefficient of $q^\chi$ by the difference between the cofficients of $q^\chi$ and $q^{-\chi}$) kills these terms, leaving only the $g=0$ term
\begin{eqnarray*}
\sum_{\chi\ge1}\(P^h_\chi(t)-P^h_{-\chi}(t)\)q^\chi &=& n^h_0(t)\,q\(1+[2]\_tq+q^2\)^{-1} \\
&=& n^h_0(t)\sum_{\chi\ge1}(-1)^{\chi-1}[\chi]\_t\;q^\chi.
\end{eqnarray*}
Recalling the definition \eqref{tnorm} of $P^h_\chi(t)$ and comparing the above formula to \eqref{vw1} gives
$$
\VW_{(0,\beta,\chi)^1}(t)\=n_0^h(t),
$$
dependent only on $\beta^2=2h-2$. By Proposition \ref{prop} this proves Theorem \ref{VWt}.

\begin{rmk}
Finally notice that substituting $q=-t^{1/2}$ into \eqref{bps} also kills all but the $g=0$ term, so the same substitution in the infinite product recovers the generating series \eqref{jac} for these invariants $\VW_{(0,\beta,\chi)^1}(t)$,
\beq{Kkv}
\sum_{h=0}^\infty n_0^h(t)u^h\=\prod_{n\ge1}\frac{1}{(1-u^n)^{20}(1-tu^n)^2(1-t^{-1}u^n)^2}\,.
\eeq

By the KKV formula \cite{kkv}, proved in \cite{KKV}, the integer Gopakumar-Vafa invariants $n_g(\beta)$\,---\,counting curves in an algebraic class $\beta$ in $S$ with $\lambda_g$-insertion, or on $X$ with the reduced obstruction theory\,---\,are governed by the exact same product formula. Since $n_g(\beta)$ also depends on $\beta$ only through its square $\beta^2=2h-2$ we write $n_g^h:=n_g(\beta)\in\Z$. Then it follows from \eqref{Kkv} that
$$
\VW_{(0,\beta,\chi)^1}(t)\=n_0^h(t)\=\sum\nolimits_{g=0}^h(-1)^g\,n^h_g\,\(t^{-1}-[2]\_t+t\)^g.
$$
I imagine this is \emph{not} a direct relationship, but is due to some duality (such as hyperk\"ahler rotation, P$\,=\,$W, mirror symmetry, or even coincidence) special to (local) $K3$ surfaces. In general one would expect the $\chi\_{-t}$ (or Nekrasov-Okounkov) refinement of genus 0 BPS invariants to precisely \emph{not} govern higher genus numerical BPS invariants, since $\chi\_{-t}$ annihilates the Jacobian of a geometric genus $g>0$ curve.
\end{rmk}

\bibliographystyle{halphanum}
\bibliography{references}

\medskip \noindent {\tt{richard.thomas@imperial.ac.uk}} \medskip

\noindent Department of Mathematics \\
\noindent Imperial College London\\
\noindent London SW7 2AZ \\
\noindent United Kingdom

\end{document}